\documentclass[11pt]{article}
\usepackage{amsmath, amsthm, amssymb, url}
\usepackage[margin=2.5cm]{geometry}
\usepackage{multirow}
\usepackage{tikz, tkz-graph, mathrsfs, enumerate}
\usepackage{authblk}

\usetikzlibrary{arrows}

\newcommand{\id}{\mathrm{id}}
\newcommand{\CA}{\mathrm{CA}}
\newcommand{\ICA}{\mathrm{ICA}}
\newcommand{\Rank}{\mathrm{Rank}}

\newcommand{\dd}{\mathrm{d}}

\newcommand{\Aut}{\mathrm{Aut}}

\theoremstyle{plain}

\newtheorem{corollary}{Corollary}
\newtheorem{lemma}{Lemma}
\newtheorem{proposition}{Proposition}
\newtheorem{theorem}{Theorem}

\theoremstyle{definition}

\newtheorem{example}{Example}
\newtheorem{remark}{Remark}
\newtheorem{question}{Question}

\begin{document}

\title{Bounding the minimal number of generators of groups and monoids of cellular automata}
\author{Alonso Castillo-Ramirez\footnote{Email: alonso.castillor@academicos.udg.mx (corresponding author)} \ and Miguel Sanchez-Alvarez\footnote{Email: miguel\_201288@hotmail.com} \\
\small{Department of Mathematics, University Centre of Exact Sciences and Engineering,\\ University of Guadalajara, Guadalajara, Mexico.} }

\maketitle

\begin{abstract}
For a group $G$ and a finite set $A$, denote by $\CA(G;A)$ the monoid of all cellular automata over $A^G$ and by $\ICA(G;A)$ its group of units. We study the minimal cardinality of a generating set, known as the \emph{rank}, of $\ICA(G;A)$. In the first part, when $G$ is a finite group, we give upper bounds for the rank in terms of the number of conjugacy classes of subgroups of $G$. The case when $G$ is a finite cyclic group has been studied before, so here we focus on the cases when $G$ is a finite dihedral group or a finite Dedekind group. In the second part, we find a basic lower bound for the rank of $\ICA(G;A)$ when $G$ is a finite group, and we apply this to show that, for any infinite abelian group $H$, the monoid $\CA(H;A)$ is not finitely generated. The same is true for various kinds of infinite groups, so we ask if there exists an infinite group $H$ such that $\CA(H;A)$ is finitely generated.         \\


\end{abstract}

\section{Introduction}\label{intro}
The theory of cellular automata (CA) has important connections with many areas of mathematics, such as group theory, topology, symbolic dynamics, coding theory, and cryptography. Recently, in \cite{CRG16a,CRG16b,CRG17}, links with semigroup theory have been explored, and, in particular, questions have been considered on the structure of the monoid of all CA and the group of all invertible CA over a given configuration space. The goal of this paper is to bound the minimal number of generators, known in semigroup theory as the \emph{rank}, of groups of invertible CA.

Let $G$ be a group and $A$ a finite set. Denote by $A^G$ the \emph{configuration space}, i.e. the set of all functions of the form $x:G \to A$. The \emph{shift action} of $G$ on $A^G$ is defined by $g \cdot x(h) = x(g^{-1}h)$, for all $x \in A^G$, $g,h \in G$. We endow $A^G$ with the \emph{prodiscrete topology}, which is the product topology of the discrete topology on $A$. A \emph{cellular automaton} over $A^G$ is a transformation $\tau : A^G \to A^G$ such that there is a finite subset $S \subseteq G$ and a function $\mu : A^S \to A$ satisfying
\[ \tau(x)(g) = \mu (( g^{-1} \cdot x) \vert_{S}), \ \ \forall x \in A^G, g \in G,  \]
where $\vert_{S}$ denotes the restriction to $S$ of a configuration in $A^G$.

Curtis-Hedlund Theorem (\cite[Theorem 1.8.1]{CSC10}) establishes that a function $\tau : A^G \to A^G$ is a cellular automaton if and only if it is continuous in the prodiscrete topology of $A^G$ and commutes with the shift action (i.e. $\tau(g \cdot x) = g \cdot \tau(x)$, for all $x \in A^G$, $g \in G$). By \cite[Corollary 1.4.11]{CSC10}, the composition of two cellular automata over $A^G$ is a cellular automaton over $A^G$. This implies that, equipped with composition of functions, the set $\CA(G;A)$ of all cellular automata over $A^G$ is a monoid. The \emph{group of units} (i.e. group of invertible elements) of $\CA(G;A)$ is denoted by $\ICA(G;A)$. When $\vert A \vert \geq 2$ and $G = \mathbb{Z}$, many interesting properties are known for $\ICA(\mathbb{Z};A)$: for example, every finite group, as well as the free group on a countable number of generators, may be embedded in $\ICA(\mathbb{Z};A)$ (see \cite{BLR88}). However, despite of several efforts, most of the algebraic properties of $\CA(G;A)$ and $\ICA(G;A)$ still remain unknown. 

Given a subset $T$ of a monoid $M$, the \emph{submonoid generated} by $T$, denoted by $\langle T \rangle$, is the smallest submonoid of $M$ that contains $T$; this is equivalent as defining $\langle T \rangle := \{ t_1 t_2 \dots t_k \in M : t_i \in T, \ \forall i, \ k \geq 0 \}$. We say that $T$ is a \emph{generating set of $M$} if $M= \langle T \rangle$. The monoid $M$ is said to be \emph{finitely generated} if it has a finite generating set. The \emph{rank} of $M$ is the minimal cardinality of a generating set:
\[ \Rank(M) := \min\{\vert T \vert : M= \langle T \rangle \}. \]	

The question of finding the rank of a monoid is important in semigroup theory; it has been answered for several kinds of transformation monoids and Rees matrix semigroups (e.g., see \cite{AS09,G14}). For the case of monoids of cellular automata over finite groups, the question has been addressed in \cite{CRG16b,CRG17}; in particular, the rank of $\ICA(G;A)$ when $G$ is a finite cyclic group has been examined in detail in \cite{CRG16a}. 

For any subset $U$ of a monoid $M$, the \emph{relative rank} of $U$ in $M$ is
\[ \Rank(M:U) = \min \{ \vert W \vert : M = \langle U \cup W \rangle  \}. \]
When $U$ is the group of units of $M$, we have the basic identity
\begin{equation}\label{rank-formula}
 \Rank(M) = \Rank(M:U) + \Rank(U),
\end{equation}
which follows as the group $U$ can only be generated by subsets of itself. The relative rank of $\ICA(G;A)$ in $\CA(G;A)$ has been established in \cite[Theorem 7]{CRG17} for finite \emph{Dedekind groups} (i.e. groups in which all subgroups are normal).   

In this paper, we study the rank of $\ICA(G;A)$ when $G$ is a finite group. In Section \ref{basic}, we introduce notation and review some basic facts, including the structure theorem for $\ICA(G;A)$ obtained in \cite{CRG17}. In Section \ref{upper}, we give upper bounds for the rank of $\ICA(G;A)$, examining in detail the cases when $G$ is a finite dihedral group or a finite Dedekind group, but also obtaining some results for a general finite group. In Section \ref{lower}, we show that, when $G$ is finite, the rank of $\ICA(G;A)$ is at least the number of conjugacy classes of subgroups of $G$. As an application, we use this to provide a simple proof that the monoid $\CA(G;A)$ is not finitely generated whenever $G$ is an infinite abelian group. This result implies that $\CA(G;A)$ is not finitely generated for various classes of infinite groups, such as free groups and the infinite dihedral group. Thus, we ask if there exists an infinite group $G$ such that the monoid $\CA(G;A)$ is finitely generated.


\section{Basic Results} \label{basic}

We assume the reader has certain familiarity with basic concepts of group theory. 

Let $G$ be a group and $A$ a finite set. The \emph{stabiliser} and \emph{$G$-orbit} of a configuration $x \in A^G$ are defined, respectively, by
\[ G_x := \{ g \in G : g \cdot x = x \} \text{ and } Gx := \{ g \cdot x : g \in G \}. \]
Stabilisers are subgroups of $G$, while the set of $G$-orbits forms a partition of $A^G$. 

Two subgroups $H_1$ and $H_2$ of $G$ are \emph{conjugate} in $G$ if there exists $g \in G$ such that $g^{-1} H_1 g = H_2$. This defines an equivalence relation on the subgroups of $G$. Denote by $[H]$ the conjugacy class of $H \leq G$. A subgroup $H \leq G$ is \emph{normal} if $[H] = \{ H \}$ (i.e. $g^{-1}H g = H$ for all $g \in G$). Let $N_G(H) := \{ g \in G : H = g^{-1} H g \} \leq G$ be the \emph{normaliser of $H$ in $G$}. Note that $H$ is always a normal subgroup of $N_G(H)$. Denote by $r(G)$ the total number of conjugacy classes of subgroups of $G$, and by $r_i(G)$ the number of conjugacy classes $[H]$ such that $H$ has index $i$ in $G$: 
\begin{align*}
r(G) & := \vert \{ [H] : H \leq G \} \vert, \\
r_i(G) & := \vert \{ [ H] : H \leq G, \ [G:H] = i \} \vert.
\end{align*}

For any $H \leq G$, denote
\[ \alpha_{[H]}(G ; A):= \vert \{ Gx \subseteq A^G :  [G_x] = [H]  \} \vert. \]
This number may be calculated using the Mobius function of the subgroup lattice of $G$, as shown in \cite[Sec. 4]{CRG17}.

For any integer $\alpha \geq 1$, let $S_{\alpha}$ be the symmetric group of degree $\alpha$. The \emph{wreath product} of a group $C$ by $S_\alpha$ is the set
\[ C \wr S_{\alpha} := \{ (v; \phi) : v \in C ^\alpha, \phi \in S_\alpha \} \]
equipped with the operation $(v;\phi) \cdot (w; \psi) = ( v  \cdot w^{\phi}; \phi \psi)$, for any $v,w \in C^\alpha, \phi, \psi \in S_\alpha$, where $\phi$ acts on $w$ by permuting its coordinates:
\[  w^\phi = (w_1, w_2, \dots, w_\alpha)^\phi := (w_{\phi(1)}, w_{\phi(2)}, \dots, w_{\phi(\alpha)}). \]
In fact, as may be seen from the above definitions, $C \wr S_{\alpha}$ is equal to the external semidirect product $C^{\alpha} \rtimes_{\varphi} S_{\alpha}$, where $\varphi : S_\alpha \to \Aut(C^{\alpha})$ is the action of $S_{\alpha}$ of permuting the coordinates of $C^{\alpha}$. For a more detailed description of the wreath product see \cite{AS09}.

\begin{theorem}[\cite{CRG17}] \label{th:ICA}
Let $G$ be a finite group and $A$ a finite set of size $q \geq 2$. Let $[H_1], \dots, [H_r]$ be the list of all different conjugacy classes of subgroups of $G$. Let $\alpha_i :=\alpha_{[H_i]}(G ; A)$. Then,
\[ \ICA(G;A) \cong \prod_{i=1}^{r} \left( (N_{G}(H_i)/H_i) \wr S_{\alpha_i} \right). \]
\end{theorem}



\section{Upper bounds for ranks} \label{upper}

The $\Rank$ function on monoids does not behave well when taking submonoids or subgroups: in other words, if $N$ is a submonoid of $M$, there may be no relation between $\Rank(N)$ and $\Rank(M)$. For example, if $M=S_n$ is the symmetric group of degree $n \geq 3$ and $N$ is a subgroup of $S_n$ generated by $\lfloor \frac{n}{2} \rfloor$ commuting transpositions, then $\Rank(S_n) = 2$, as $S_n$ may be generated by a transposition and an $n$-cycle, but $\Rank(N) = \lfloor \frac{n}{2} \rfloor$. It is even possible that $M$ is finitely generated but $N$ is not finitely generated (such as the case of the free group on two symbols and its commutator subgroup). However, the following lemma gives us some elementary tools to bound the rank in some cases.

\begin{lemma}\label{le:basic}
Let $G$ and $H$ be a groups, and let $N$ be a normal subgroup of $G$. Then:
\begin{enumerate}
\item $\Rank(G/N) \leq \Rank(G)$.
\item $\Rank(G \times H) \leq \Rank(G) + \Rank(H)$.
\item $\Rank( G \wr S_\alpha) \leq \Rank(G) + \Rank(S_\alpha)$, for any $\alpha \geq 1$.
\item $\Rank(\mathbb{Z}_d \wr S_\alpha) = 2$, for any $d, \alpha \geq 2$. . 
\end{enumerate}
\end{lemma}
\begin{proof}
Parts 1 and 2 are straightforward. For parts 3 and 4, see \cite[Corollary 5]{CRG17} and \cite[Lemma 5]{CRG16a}, respectively.
\end{proof}

We shall use Lemma \ref{le:basic} together with Theorem \ref{th:ICA} in order to find upper bounds for $\ICA(G;A)$. Because of part 3 in Lemma \ref{le:basic}, it is now relevant to determine some values of the $\alpha_i$'s that appear in Theorem \ref{th:ICA}.

\begin{lemma} \label{alpha}
Let $G$ be a finite group and $A$ a finite set of size $q \geq 2$. Let $H$ be a subgroup of $G$.
\begin{enumerate}
\item $\alpha_{[G]}(G;A) = q$. 
\item $\alpha_{[H]}(G;A) = 1$ if and only if $[G : H] = 2$ and $q=2$.
\item If $q \geq 3$, then $\alpha_{[H]}(G;A)  \geq 3$. 
\end{enumerate} 
\end{lemma}
\begin{proof}
Parts 1 and 2 correspond to Remark 1 and Lemma 5 in \cite{CRG17}, respectively. For part 2, Suppose that $q \geq 3$ and $\{0,1,2 \} \subseteq A$. Define configurations $z_1, z_2, z_3 \in A^G$ as follows,
\[ 
z_1 (g)  = \begin{cases}
1 & \text{if } g \in H \\
0 & \text{if } g \not\in H, 
\end{cases} \ \ 
z_2 (g)  = \begin{cases}
2 & \text{if } g \in H \\
0 & \text{if } g \not\in H, 
\end{cases} \ \ 
z_3 (g)  = \begin{cases}
1 & \text{if } g \in H \\
2 & \text{if } g \not\in H, 
\end{cases} \ \ 
\]
All three configurations are in different orbits and  $G_{z_i} = H$, for $i=1,2,3$. Hence $\alpha_{[H]}(G;A) \geq 3$. 
\end{proof}

Although we shall not use explicitly part 3 of the previous lemma, the result is interesting as it shows that, for $q \geq 3$, our upper bounds cannot be refined by a more careful examination of the values of the $\alpha_i's$, as, for all $\alpha \geq 3$, we have $\Rank(S_\alpha) = 2$. 


\subsection{Dihedral groups} \label{dihedral}

In this section we investigate the rank of $\ICA(D_{2n};A)$, where $D_{2n}$ is the dihedral group of order $2n$, with $n \geq 1$, and $A$ is a finite set of size $q \geq 2$. We shall use the following standard presentation of $D_{2n}$:

\[ D_{2n} = \left\langle \rho, s \; \vert \; \rho^n = s^2 = s\rho s \rho = \id \right\rangle. \]

\begin{lemma}\label{rk-dihedral}
For any $n \geq 2$ and $\alpha \geq 2$,  $\Rank(D_{2n} \wr S_\alpha) \leq 3$. 
\end{lemma}
\begin{proof}
By \cite[Lemma 5]{CRG16a}, we know that $\langle \rho \rangle \wr S_\alpha \cong \mathbb{Z}_n \wr S_\alpha$ may be generated by two elements. Hence, by adding $((s, \id, \dots, \id);\id)$ we may generate the whole $D_{2n} \wr S_\alpha$ with three elements. 
\end{proof}

Given a subgroup $H \leq D_{2n}$, we shall now analyze the quotient group $N_{G}(H)/H$.

\begin{lemma}\label{le-di1}
Let $G =  D_{2n}$ and let $H \leq D_{2n}$ be a subgroup of odd index $m$. Then $H$ is self-normalizing, i.e. $N_{G}(H) = H$.
\end{lemma}
\begin{proof}
By \cite[Theorem 3.3]{Conrad}, all subgroups of $D_{2n}$ with index $m$ are conjugate to each other. By \cite[Corollary 3.2]{Conrad}, there are $m$ subgroups of $D_{2n}$ of index $m$, so $\vert [H] \vert = m = [D_{2n} : H ] $. On the other hand, by the Orbit-Stabilizer Theorem applied to the conjugation action of $D_{2n}$ on its subgroups we have $\vert [H ] \vert = [D_{2n}: N_{G}(H) ]$. Therefore, $[D_{2n} : H ]=[D_{2n}: N_{G}(H) ]$ and $N_{G}(H) = H$. 
\end{proof}

\begin{lemma}\label{le-di2}
Let $G = D_{2n}$ and let $H \leq D_{2n}$ be a proper subgroup of even index $m$. Then $H$ is normal in $D_{2n}$ and $D_{2n}/H \cong D_{m}$, except when $n$ is even, $m \mid n$, and $[H] = [\langle \rho^m, s \rangle]$ or $[H]= [\langle \rho^m, \rho s \rangle]$, in which case $N_{G}(H) / H \cong \mathbb{Z}_2$.  
\end{lemma}
\begin{proof}
We shall use Corollary 3.2 and Theorem 3.3 in \cite{Conrad}. There are two cases to consider:
\begin{enumerate}
\item Suppose $n$ is odd. Then, $D_{2n}$ has a unique subgroup of index $m$, so $1 = \vert [H] \vert = [D_{2n} : N_{G}(H)]$. This implies that $D_{2n} = N_{G}(H)$, so $H$ is normal in $D_{2n}$.  

\item Suppose that $n$ is even. If $m \nmid n$, $D_{2n}$ has a unique subgroup of index $m$, so $H$ is normal by the same argument as in the previous case. If $m \mid n$, then $D_{2n}$ has $m+1$ subgroups with index $m$ partitioned into $3$ conjugacy classes $[\langle \rho^{m/2} \rangle]$,  $[\langle \rho^m, s \rangle]$ and $[\langle \rho^m, rs \rangle]$ of sizes $1$, $\frac{m}{2}$ and $\frac{m}{2}$, respectively. If $\vert [H] \vert = 1$, again $H$ is normal. If $\vert [H] \vert = \frac{m}{2}$, then $[D_{2n} : N_{G}(H) ] = \frac{1}{2} [D_{2n} : H]$. Hence, $[N_{G}(H) : H] = 2$, and $N_{G}(H) / H \cong \mathbb{Z}_2$.  
\end{enumerate}
The fact that $D_{2n}/H \cong D_{m}$ whenever $H$ is normal follows by \cite[Theorem 2.3]{Conrad}. 
\end{proof}

Let $\dd(n)$ be number of divisors of $n$, including $1$ and $n$ itself. Let $\dd_-(n)$ and $\dd_+(n)$ be the number of odd and even divisors of $n$, respectively.

When $n$ is odd, $D_{2n}$ has exactly $1$ conjugacy class of subgroups of index $m$, for every $m \mid 2n$; hence, if $n$ is odd, $r(D_{2n}) = \dd(2n)$. When $n$ is even, $D_{2n}$ has exactly $1$ conjugacy class of index $m$, when $m \mid 2n$ is odd or $m \nmid n$, and exactly $3$ conjugacy classes when $m \mid 2n$ is even and $m \mid n$; hence, if $n$ is even, $r(D_{2n}) = \dd(2n) + 2\dd_+(n)$.

\begin{theorem} \label{dihedral}
Let $n \geq 3$ be an integer and $A$ a finite set of size at least $2$. 
\[  \Rank(\ICA(D_{2n};A)) \leq
\begin{cases}
 2 \dd_-(2n)  + 3 \dd_+ (2n) - 3 & \text{if $n$ is odd and $q = 2$,} \\
 2 \dd_-(2n)  + 3 \dd_+ (2n) - 1& \text{if $n$ is odd and $q \geq 3$,} \\
  2\dd_- (2n) + 3\dd_+(2n) + 2 \dd_+(n) - 3 & \text{if $n$ is even and $q =2$,} \\
  2\dd_- (2n) + 3\dd_+(2n) + 4 \dd_+(n) - 1 & \text{if $n$ is even and $q \geq 3$.} 
\end{cases} \]
\end{theorem} 
\begin{proof}
Let $[H_1], \dots, [H_r]$ be the conjugacy classes of subgroups of $D_{2n}$. Let $m_i = [G:H_i]$, with $m_{r-1} =2$ and $m_r = 1$ (i.e. $H_r = D_{2n}$). Define $\alpha_i := \alpha_{[H_i]}(G;A)$. 

Let $n$ be odd. Then, by Theorem \ref{th:ICA}, and Lemmas \ref{le-di1} and \ref{le-di2},
\[ \ICA(D_{2n};A) \cong \prod_{m_i \mid 2n \text{ odd}} S_{\alpha_i} \times \prod_{2<m_i \mid 2n \text{ even}} ( D_{m_i} \wr  S_{\alpha_i} ) \times (\mathbb{Z}_2 \wr S_{\alpha_{r-1}})\]
By Lemmas \ref{le:basic} and \ref{rk-dihedral},
\begin{align*}
\Rank(\ICA(D_{2n};A)) & \leq  \sum_{m_i \mid 2n \text{ odd}} \Rank(S_{\alpha_i}) + \sum_{2< m_i \mid 2n \text{ even}} \Rank( D_{m_i} \wr  S_{\alpha_i} ) + 2 \\
& \leq  2 \dd_-(2n) + 3 (\dd_+(2n) - 1)  + 2 \\ 
& =  2 \dd_-(2n)  + 3 \dd_+ (2n) - 1. 
\end{align*}
When $q=2$, Lemma \ref{alpha} shows that $\alpha_r = 2$ and $\alpha_{r-1} = 1$, so $S_{\alpha_r} \cong S_2$ and $\mathbb{Z}_2 \wr S_{\alpha_{r-1}} \cong \mathbb{Z}_2$. Therefore,
\begin{align*}
\Rank(\ICA(D_{2n};A)) & \leq  1 + \sum_{1 < m_i \mid 2n \text{ odd}} 2 + \sum_{2< m_i \mid 2n \text{ even}} 3 + 1 \\
& \leq  2 (\dd_-(2n) - 1) + 3 (\dd_+(2n) - 1)  + 2 \\ 
& =  2 \dd_-(2n)  + 3 \dd_+ (2n)  - 3. 
\end{align*}

Now let $n$ be even. Then, 
\begin{align*}
\ICA(D_{2n};A) & \cong \prod_{m_i \mid 2n \text{ odd}} S_{\alpha_i} \times \prod_{2< m_i \mid 2n \text{ even}} ( D_{m_i} \wr  S_{\alpha_i} ) \\ 
& \quad \times (\mathbb{Z}_2 \wr S_{\alpha_{r-1}}) \times \prod_{m_i \mid n \text{ even }} ( (\mathbb{Z}_2 \wr S_{\alpha_i}) \times  (\mathbb{Z}_2 \wr S_{\alpha_{i}}))
\end{align*}
Hence, 
\begin{align*}
\Rank(\ICA(D_{2n};A)) & \leq \sum_{m_i \mid 2n \text{ odd}} 2 + \sum_{2< m_i \mid 2n \text{ even}} 3 + 2 + \sum_{m_i \mid n \text{ even}} 4 \\
& =  2\dd_- (2n) + 3( \dd_+(2n) - 1) + 2 + 4 \dd_+(n) \\
& = 2\dd_- (2n) + 3\dd_+(2n) + 4 \dd_+(n) - 1 .
\end{align*}
When $q=2$, by Lemma \ref{alpha} we have $S_{\alpha_r} \cong S_2$, $\mathbb{Z}_2 \wr S_{\alpha_{r-1}} \cong \mathbb{Z}_2$ and $\mathbb{Z}_2 \wr S_{\alpha_{i}} \cong \mathbb{Z}_2$, so
\begin{align*}
\Rank(\ICA(D_{2n};A)) & \leq 1 + \sum_{1 < m_i \mid 2n \text{ odd}} 2 + \sum_{2< m_i \mid 2n \text{ even}} 3 + 1 + \sum_{m_i \mid n \text{ even}} 2 \\
& =  2(\dd_- (2n)-1) + 3( \dd_+(2n) - 1) + 2 \dd_+(n) + 2 \\
& = 2\dd_- (2n) + 3\dd_+(2n) + 2 \dd_+(n) - 3 .
\end{align*}

\end{proof}

\begin{example}
Let $A$ be a finite set of size $q \geq 2$. By the previous theorem,
\[ \Rank(\ICA(D_6; A)) \leq  \begin{cases} 
 2 \dd_-(6)  + 3 \dd_+ (6) - 3 = 2\cdot 2 + 3 \cdot 2 - 3 = 7  & \text{ if } q = 2, \\
2d_-(6) + 3d_+(6) - 1 = 2 \cdot 2 + 3 \cdot 2 - 1 = 9 & \text{ if } q \geq 2.
\end{cases}\]
On the other hand,
\[ \Rank(\ICA(D_8; A)) \leq \begin{cases}
 2\dd_- (8) + 3\dd_+(8) + 2 \dd_+(4) - 3 = 12 & \text{ if } q = 2, \\
  2d_-(8) + 3d_+(8) + 4d_+(4) - 1 = 18 & \text{ if } q \geq 2.  
\end{cases} \]
\end{example}


\subsection{Other finite groups}

Recall that $r(G)$ denotes the total number of conjugacy classes of subgroups of $G$ and $r_i(G)$ the number of conjugacy classes $[H]$ such that $H$ has index $i$ in $G$. The following results are an improvement of \cite[Corollary 5]{CRG17}.

\begin{theorem}\label{cor:bound}
Let $G$ be a finite Dedekind group and $A$ a finite set of size $q \geq 2$. Let $r:=r(G)$ and $r_i := r_i(G)$. Let $p_1, \dots, p_s$ be the prime divisors of $\vert G \vert$ and define $r_P :=\sum_{i=1}^s r_{p_i}$. Then,
\[ \Rank(\ICA(G;A))  \leq \begin{cases}
(r - r_P  - 1) \Rank(G) + 2 r - r_2 - 1,  & \text{ if } q=2, \\
(r - r_P  - 1) \Rank(G) + 2 r, & \text{ if } q \geq 3.
\end{cases} \]
\end{theorem}
\begin{proof}
Let $H_1, H_2, \dots, H_r$ be the list of different subgroups of $G$ with $H_r = G$. If $H_i$ is a subgroup of index $p_k$, then $(G/H_i)\wr S_{\alpha_i} \cong  \mathbb{Z}_{p_k} \wr S_{\alpha_i}$ is a group with rank $2$, by Lemma \ref{le:basic}. Thus, by Theorem \ref{th:ICA} we have:
\begin{align*}
\Rank(\ICA(G;A)) & \leq  \sum_{i=1}^{r-1} \Rank((G/H_i)\wr S_{\alpha_i}) + \Rank(S_q)  \\ 
& \leq \sum_{[G:H_i] = p_k} 2 + \sum_{[G:H_i] \neq p_k } (\Rank(G) + 2) + 2\\
& =  2 r_P +  (r - r_P - 1 )(\Rank(G) + 2) + 2 \\ 
& =  (r - r_P  - 1) \Rank(G) + 2r .  
\end{align*}
If $q=2$, we may improve this bound by using Lemma \ref{alpha}:
\begin{align*}
\Rank(\ICA(G;A)) & \leq \sum_{[G : H_i]=2}\Rank((G/H_i)\wr S_1) + \sum_{[G : H_i] = p_k \neq 2}\Rank((G/H_i)\wr S_{\alpha_i}) \\ 
& + \sum_{1 \neq [G:H_i] \neq p_k} \Rank((G/H_i)\wr S_{\alpha_i}) + \Rank(S_2) \\  
& = r_2 + 2(r_P - r_2) +  (r - r_P - 1 )(\Rank(G) + 2)  + 1   \\
& =   (r - r_P  - 1) \Rank(G) + 2r - r_2 - 1.  
\end{align*}

 \end{proof}

\begin{example}
The smallest example of a nonabelian Dedekind group is the quaternion group 
\[ Q_8 = \langle x,y \; \vert \; x^4 = x^2 y^{-2} = y^{-1}xy x  = \id  \rangle, \]
which has order $8$. It is generated by two elements, and it is noncyclic, so $\Rank(Q_8) = 2$. Moreover, $r = r(Q_8) = 6$ and, as $2$ is the only prime divisor of $8$, we have $r_P  = r_2 = 3$. Therefore,
\[ \Rank(\ICA(Q_8;A))  \leq \begin{cases}
(6 - 3 - 1) \cdot 2 + 2 \cdot 6 - 3 - 1 = 12,  & \text{if } q=2, \\
(6 - 3  - 1) \cdot 2 + 2 \cdot 6 = 16, & \text{if } q \geq 3.
\end{cases} \]
\end{example}
 
 \begin{corollary}
 Let $G$ be a finite Dedekind group and $A$ a finite set of size $q \geq 2$. With the notation of Theorem \ref{cor:bound},
\[ \Rank(\CA(G;A))  \leq \begin{cases}
(r - r_P  - 1) \Rank(G) + \frac{1}{2} r (r+5) - 2r_2 - 1,  & \text{if } q=2 \\
(r - r_P  - 1) \Rank(G) + \frac{1}{2} r (r+5), & \text{otherwise.}
\end{cases} \]
 \end{corollary}
 \begin{proof}
 The result follows by Theorem \ref{cor:bound}, identity (\ref{rank-formula}) and the basic upper bound for the relative rank that follows from \cite[Theorem 7]{CRG17}:
 \[ \Rank(\CA(G;A):\ICA(G;A)) \leq \begin{cases} 
 \binom{r}{2} + r - r_2 & \text{if } q=2 \\
 \binom{r}{2} + r, & \text{otherwise.}
 \end{cases}  \]

 \end{proof}

Now focus now when $G$ is not necessarily a Dedekind group. 

\begin{lemma}\label{le-aux1}
Let $G$ be a finite group and $H$ a subgroup of $G$ of prime index $p$. Let $A$ be a finite set of size $q \geq 2$ and $\alpha := \alpha_{[H]}(G;A)$. Then
\[ \Rank\left( (N_{G}(H)/H) \wr S_{\alpha} \right) \leq \begin{cases}
1 & \text{if } p=2 \text{ and } q=2\\
2 & \text{otherwise}.
\end{cases} \] 
\end{lemma}
\begin{proof}
By Lagrange's theorem, $N_{G}(H)=H$ or $N_{G}(H)=G$. Hence, in order to find an upper bound for the above rank, we assume that $H$ is normal in $G$. As the index is prime, $G/H \cong \mathbb{Z}_p$. If $p=2$ and $q=2$, Lemma \ref{alpha} shows that $\alpha = 1$, so $\Rank(\mathbb{Z}_2 \wr S_{1})  = 1$. For the rest of the cases we have that $\Rank(\mathbb{Z}_p \wr S_{\alpha}) = 2$, by Lemma \ref{le:basic}.

\end{proof}

The \emph{length} of $G$ (see \cite[Sec. 1.15]{C94}) is the length $\ell := \ell(G)$ of the longest chain of proper subgroups
\[ 1=G_0 < G_1 < \dots < G_\ell = G. \]
The lengths of the symmetric groups are known by \cite{CST89}: $\ell(S_n) = \lceil 3n/2 \rceil - b(n)-1$, where $b(n)$ is the numbers of ones in the base $2$ expansion of $n$. As, $\ell(G) = \ell(N) + \ell(G/N)$ for any normal subgroup $N$ of $G$, the length of a finite group is equal to the sum of the lengths of its compositions factors; hence, the question of calculating the length of all finite groups is reduced to calculating the length of all finite simple groups. Moreover, $\ell(G) \leq \log_2(\vert G \vert)$ (see \cite[Lemma 2.2]{CST89}).  

\begin{lemma}\label{le-aux2}
Let $G$ be a finite group and $H$ a subgroup of $G$. Let $A$ be a finite set of size $q \geq 2$ and $\alpha := \alpha_{[H]}(G;A)$. Then,
\[ \Rank\left( (N_{G}(H)/H) \wr S_{\alpha} \right) \leq \ell(G) + 2 \] 
\end{lemma}
\begin{proof}
By Lemma \ref{le:basic}, $\Rank\left( (N_{G}(H)/H) \wr S_{\alpha} \right) \leq \Rank(N_{G}(H)) + 2$. Observe that $\Rank(G) \leq \ell(G)$, as the set $\{ g_i \in G : g_i \in G_i - G_{i-1}, \ i=1,\dots, \ell \}$ (with $G_i$ as the above chain of proper subgroups) generates $G$. Moreover, it is clear that $\ell(K) \leq \ell(G)$ for every subgroup $K \leq G$, so the result follows by letting $K=N_{G}(H)$. 
\end{proof}

\begin{theorem}
Let $G$ be a finite group of size $n$, $r:= r(G)$, and $A$ a finite set of size $q \geq 2$. Let $r_i$ be the number of conjugacy classes of subgroups of $G$ of index $i$. Let $p_1, \dots, p_s$ be the prime divisors of $\vert G \vert$ and let $r_P = \sum_{i=1}^s r_i$. Then:
\[ 
\Rank( \ICA( G; A )) \leq \begin{cases}
 (r-r_P -1) \ell(G) + 2r - r_2 -1 & \text{if } q=2, \\
(r - r_P - 1) \ell(G)  + 2r  & \text{if } q\geq 3.
\end{cases} \]
\end{theorem}
\begin{proof}
Let $H_1, H_2, \dots, H_r$ be the list of different subgroups of $G$ with $H_r = G$. By Theorem \ref{th:ICA} and Lemmas \ref{le:basic}, \ref{le-aux1}, \ref{le-aux2},
\begin{align*}
\Rank( \ICA(G; A )) & \leq \sum_{i=1}^{r-1} \Rank\left( (N_{G}(H_i)/H_i) \wr S_{\alpha_i} \right) + \Rank(S_q) \\
& \leq \sum_{[G : H_i]=p_k} 2 +  \sum_{1 \neq [G : H_i] \neq p_k} (\ell(G) + 2) + 2 \\
& = 2 r_P+ (r - r_P - 1) (\ell(G) + 2) + 2  \\
& =  (r - r_P - 1) \ell(G)  + 2r .
\end{align*}
When $q=2$, we may improve this bound as follows:
\begin{align*}
\Rank( \ICA(G; A )) & \leq   \sum_{[G : H_i]=2} 1  +  \sum_{[G : H_i]=p_k \neq 2} 2   +  \sum_{1<[G : H_i] \neq p_k} (\ell(G) + 2) + 1 \\
& = r_2  + 2(r_P - r_2) + (r - r_P - 1) (\ell(G)+2) + 1  \\
& = (r-r_P -1) \ell(G) + 2r - r_2 - 1.
\end{align*}

\end{proof}

If $G$ is a subgroup of $S_n$, we may find a good upper bound for $\Rank( \ICA( G; A ))$ in terms of $n$ by using a theorem of McIver and Neumann.

\begin{proposition}
Suppose that $G \leq S_n$, for some $n >3$. Let $r := r(G)$. Then 
\[  \Rank( \ICA( G; A )) \leq \begin{cases}
 (r -1) \left\lfloor \frac{n}{2} \right\rfloor + 2r - r_2 -1 & \text{if } q=2, \\
(r  - 1) \left\lfloor \frac{n}{2} \right\rfloor  + 2r  & \text{if } q \geq 3.
\end{cases} \]
\end{proposition}
\begin{proof}
By \cite{MN87}, for every $n > 3$ and every $K \leq S_n$, $\Rank(K) \leq \lfloor \frac{n}{2} \rfloor$. The rest of the proof is analogous to the previous one.
\end{proof}

\begin{example}
Consider the symmetric group $S_4$. In this case it is known that $r=r(S_4) = 11$ and $r_2 = 1$ (as $A_4$ is its only subgroup of index $2$). Therefore,
\[  \Rank( \ICA( S_4; A )) \leq \begin{cases}
 (11 -1) \frac{4}{2} + 2 \cdot 11  - 1 -1 = 40  & \text{if } q=2, \\
(11  - 1) \frac{4}{2} + 2 \cdot 11 = 42 & \text{if } q \geq 3.
\end{cases} \]
For sake of comparison, the group $\ICA( S_4; \{0,1 \} )$ has order $2^{2^{24}}$.
\end{example}


\section{Lower bounds on ranks} \label{lower}

\subsection{Finite groups}

\begin{proposition}\label{lower-bound}
Let $G$ be a finite group and $A$ a finite set of size $q \geq 2$. Then
\[ \Rank(\ICA(G;A) \geq \begin{cases} 
r(G) -  r_2(G)  & \text{if } q=2, \\
r(G)  & \text{otherwise}. 
\end{cases} . \]
\end{proposition} 
\begin{proof}
Let $[H_1], [H_2], \dots, [H_r]$ be the conjugacy clases of subgroups of $G$, with $r=r(G)$. As long as $\alpha_i > 1$, the factor $(N_{G}(H_i)/H_i) \wr S_{\alpha_i}$, in the decomposition of $\ICA(G;A)$, has a proper normal subgroup $ (N_{G}(H_i)/H_i) \wr A_{\alpha_i}$ (where $A_{\alpha_i}$ is the alternating group of degree $\alpha_i$). We know that $\alpha_i =1$ if and only if $[G:H]=2$ and $q=2$ (Lemma \ref{alpha}). Hence, for $q \geq 3$, we have
\[  \Rank(\ICA(G;A)) \geq \Rank\left( \frac{\prod_{i=1}^{r} \left( (N_{G}(H_i)/H_i) \wr S_{\alpha_i} \right)}{ \prod_{i=1}^{r} \left( (N_{G}(H_i)/H_i) \wr A_{\alpha_i} \right) } \right) 
= \Rank\left(\prod_{i=1}^{r} \mathbb{Z}_2\right) = r.  \]    

Assume now that $q=2$, and let $[H_1], \dots, [H_{r_2}]$ be the conjugacy classes of subgroups of index two, with $r_2 = r_2(G)$. Now, $\Rank(\ICA(G;A))$ is at least
\[ \Rank\left(\frac{\prod_{i=1}^{r} \left( (N_{G}(H_i)/H_i) \wr S_{\alpha_i} \right)}{ \prod_{i=1}^{r} \left( (N_{G}(H_i)/H_i) \wr A_{\alpha_i} \right) } \right) =  \Rank\left(\prod_{i=r_2+1}^r \mathbb{Z}_2 \right) = r - r_2,  \]
and the result follows. 
\end{proof}

The previous result could be refined for special classes of finite groups. In \cite{CRG16a} this has been done for cyclic groups, and we do it next for dihedral groups. 

\begin{proposition}
Let $n \geq 1$ and $A$ a finite set of size $q \geq 2$. 
\[  \Rank(\ICA(D_{2n};A)) \geq
\begin{cases}
  \dd_-(2n) + 2\dd_+(2n)& \text{if $n$ is odd and $q \geq 3$,} \\
  \dd_-(2n) + 2\dd_+(2n) -1 &  \text{if $n$ is odd and $q = 2$,} \\
     \dd_-(2n) + 2\dd_+(2n) + 4 \dd_+(n)  & \text{if $n$ is even and $q \geq 3$,}\\
     \dd_-(2n) + 2\dd_+(2n) + 2 \dd_+(n) - 1  & \text{if $n$ is even and $q = 2$,}
\end{cases} \]
\end{proposition}
\begin{proof}
We shall use the decomposition of $\ICA(D_{2n};A)$ given in the proof of Theorem \ref{dihedral}. For each $m_i \mid 2n$ even greater than $2$, the corresponding $\alpha_i$ is greater than $1$ by Lemma \ref{alpha}. The group $D_{m_i} \wr S_{\alpha_i}$ has a normal subgroup $N \cong (\mathbb{Z}_{m_i /2})^{\alpha_i}$ such that $(D_{m_i} \wr S_{\alpha_i}) / N \cong \mathbb{Z}_2 \wr S_{\alpha_i}$. Now, $\mathbb{Z}_2 \wr S_{\alpha_i}$ has a normal subgroup   
\[ U = \left\{ ((a_1, \dots, a_{\alpha_i}); \id) : \sum_{j=1}^{\alpha_i} a_j = 0 \mod(2) \right\} \]
such that $(\mathbb{Z}_{2} \wr S_{\alpha_i})/U \cong \mathbb{Z}_2 \times S_{\alpha_i}$. Finally, a copy of the alternating group $A_{\alpha_i}$ is a normal subgroup of $\mathbb{Z}_2 \times S_{\alpha_i}$ with quotient group $\mathbb{Z}_2 \times \mathbb{Z}_2$. This implies that $D_{m_i} \wr S_{\alpha_i}$ has a normal subgroup with quotient group isomorphic to $\mathbb{Z}_2 \times \mathbb{Z}_2$. 

Suppose that $n$ is odd and $q \geq 3$. Then $\ICA(D_{2n};A)$ has a normal subgroup with quotient group isomorphic to
\[ \prod_{m_i \mid 2n \text{ odd }} \mathbb{Z}_2 \times \prod_{2 < m_i \mid 2n \text{ even}} (\mathbb{Z}_2)^2 \times (\mathbb{Z}_2)^2. \]
Thus, $\dd_-(2n) + 2\dd_+(2n) \leq \Rank(\ICA(D_{2n};A))$. If $q = 2$, the last factor above becomes just $\mathbb{Z}_2$, as $\alpha_i = 1$ here, and the result follows.  

Suppose that $n$ is even and $q \geq 3$. Then $\ICA(D_{2n};A)$ has a normal subgroup with quotient group isomorphic to
\[ \prod_{m_i \mid 2n \text{ odd}} \mathbb{Z}_2 \times \prod_{2 < m_i \mid 2n \text{ even}} (\mathbb{Z}_2)^2 \times (\mathbb{Z}_2)^2  \times \prod_{m_i \mid n \text{ even}} (\mathbb{Z}_2)^4 \]
Therefore, $\dd_-(2n) + 2\dd_+(2n) + 4 \dd_+(n) \leq  \Rank(\ICA(D_{2n};A))$. If $q=2$, the last $d_+(n) + 1$ factors become $\mathbb{Z}_2 \times \prod_{m_i \mid n \text{ even}} (\mathbb{Z}_2)^2 $ and the result follows. 
\end{proof}


\subsection{Infinite groups}

Now we turn our attention to the case when $G$ is an infinite group. It was shown in \cite{BLR88} that $\ICA(\mathbb{Z};A)$ (and so $\CA(\mathbb{Z};A)$) is not finitely generated by studying its action on periodic configurations. In this section, using elementary techniques, we prove that the monoid $\CA(G;A)$ is not finitely generated when $G$ is infinite abelian, free or infinite dihedral; this illustrates an application of the study of ranks of groups of cellular automata over finite groups. 

\begin{remark}\label{remark}
Let $G$ be a group that is not finitely generated. Suppose that $\CA(G;A)$ has a finite generating set $H=\{ \tau_1, \dots, \tau_k \}$. Let $S_i$ be a memory set for each $\tau_i$. Then $G \neq \langle \cup_{i=1}^k S_i \rangle$, so let $\tau \in \CA(G;A)$ be such that its minimal memory set is not contained in $\langle \cup_{i=1}^k S_i \rangle$. As memory set for the composition $\tau_i \circ \tau_j$ is $S_i S_j = \{ s_i s_j : s_i \in S_i, s_j \in S_j \}$, $\tau$ cannot be in the monoid generated by $H$, contradicting that $H$ is a generating set for $\CA(G;A)$. This shows that $\CA(G;A)$ is not finitely generated whenever $G$ is not finitely generated.  
\end{remark}

The next result, which holds for an arbitrary group $G$, will be our main tool.

\begin{lemma}\label{th:quotient}
Let $G$ be a group and $A$ a set. For every normal subgroup $N$ of $G$, 
\[ \Rank(\CA(G/N;A))  \leq \Rank(\CA(G;A))  . \] 
\end{lemma}
\begin{proof}
By \cite[Proposition 1.6.2]{CSC10}, there is a monoid epimorphism $\Phi : \CA(G;A) \to \CA(G/N;A)$. Hence, the image under $\Phi$ of a generating set for $\CA(G;A)$ of minimal size is a generating set for $\CA(G/N;A)$ (not necesarily of minimal size).  
\end{proof}

\begin{theorem}
Let $G$ be an infinite abelian group and $A$ a finite set of size $q \geq 2$. Then, the monoid $\CA(G;A)$ is not finitely generated.
\end{theorem}
\begin{proof}
If $G$ is not finitely generated, then Remark \ref{remark} shows that $\CA(G;A)$ is not finitely generated, so assume that $G$ is finitely generated. By the Fundamental Theorem of Finitely Generated Abelian Groups, $G$ is isomorphic to
\[ \mathbb{Z}^s \oplus \mathbb{Z}_{p_1} \oplus \mathbb{Z}_{p_2} \oplus \dots \oplus \mathbb{Z}_{p_t},   \]
where $s \geq 1$ (because $G$ is infinite), and $p_1, \dots, p_t$ are powers of primes. Then, for every $k \geq 1$, we may find a subgroup 
\[ N  \cong \langle 2^k \rangle \oplus \mathbb{Z}^{s-1} \oplus \mathbb{Z}_{p_1} \oplus \mathbb{Z}_{p_2} \oplus \dots \oplus \mathbb{Z}_{p_t} \]
such that $G/N \cong \mathbb{Z}_{2^k}$. By Lemma \ref{th:quotient} and Proposition \ref{lower-bound},
\[ \Rank(\CA(G;A)) \geq \Rank(\CA(\mathbb{Z}_{2^k};A))  \geq \Rank(\ICA(\mathbb{Z}_{2^k};A)) \geq r(\mathbb{Z}_{2^k}) - 1 = k. \]
As the above holds for every $k \geq 1$, then $\CA(G;A)$ is not finitely generated.  
\end{proof}

The \emph{abelianization} of any group $G$ is the quotient $G/[G,G]$, where $[G,G]$ is its \emph{commutator subgroup}, i.e. the normal subgroup of $G$ generated by all commutators $[g,h]:=ghg^{-1}h^{-1}$, $g,h \in G$. The abelianization of $G$ is in fact the largest abelian quotient of $G$. 

\begin{corollary}
Let $G$ be a group with an infinite abelianization and $A$ a finite set of size $q \geq 2$. Then, the monoid $\CA(G;A)$ is not finitely generated.
\end{corollary}
\begin{proof}
Let $G^\prime = G/[G,G]$ be the abelianization of $G$. By Lemma \ref{th:quotient}, we have $\Rank(\CA(G^\prime;A))  \leq \Rank(\CA(G;A))$. But $\CA(G^\prime;A)$ is not finitely generated by the previous theorem, so the result follows.  
\end{proof}

\begin{corollary}
Let $F_S$ be a free group on a set $S$ and $A$ a finite set of size $q \geq 2$. Then, the monoid $\CA(F_S;A)$ is not finitely generated.
\end{corollary}
\begin{proof}
As $F_S$ has an infinite abelianization, which is the free abelian group on $S$, the result follows by the previous corollary. 
\end{proof}

The infinite dihedral group $D_\infty = \langle x,y \; \vert \; x^2 = y^2 = \id \rangle$ has finite abelianization $\mathbb{Z}_2 \oplus \mathbb{Z}_2$. However, we can still show that $\CA(D_\infty, A)$ is not finitely generated.

\begin{proposition}
Let $A$ be a finite set of size $q \geq 2$. Then, $\CA(D_\infty ; A)$ is not finitely generated.
\end{proposition}
\begin{proof}
For every $n \geq 1$, define $H_n := \langle (xy)^n \rangle \leq D_\infty$, which is a normal subgroup of $D_\infty$ with quotient group $D_\infty / H_n  \cong D_{2n}$. By Proposition \ref{lower-bound}, As $r_2(D_n) = 1$, for every $n \geq 1$,
\[ \Rank(\CA(D_\infty ; A)) \geq \Rank(\CA(D_{2n}; A)) \geq r(D_{2n}) - 1, \]
We know that $r(D_{2n}) \geq \dd(2n)$, so, taking $n=2^{k-1}$, $k \geq 1$, we see that $\Rank(\CA(D_\infty; A)) \geq d(2^k) -1 \geq k$, for any $k \geq 1$. 
\end{proof}

\begin{question}
Is there an infinite group $G$ such that $\CA(G;A)$ is finitely generated? 
\end{question} 

The techniques of this section seem ineffective to answer this for infinite groups with few proper quotients, such as the infinite symmetric group.


\section{Acknowledgments}

The second author thanks the National Council of Science and Technology (CONACYT) of the Government of Mexico for the National Scholarship (No. 423151) which allowed him to do part of the research reported in this article.    



\begin{thebibliography}{}

\bibitem{AS09} Ara\'ujo, J., Schneider, C.: The rank of the endomorphism monoid of a uniform partition. Semigroup Forum \textbf{78}, 498--510 (2009).

\bibitem{BLR88} Boyle, M., Lind, D., Rudolph, D.: The Automorphism Group of a Shift of Finite Type. Trans. Amer. Math. Soc. \textbf{306}, no. 1, (1988). 

\bibitem{C94} Cameron, P.J.: Permutation Groups. London Mathematical Society Student Texts \textbf{45}, Cambridge University Press, 1999.

\bibitem{CST89} Cameron, P.J., Solomon, R., Turull, A.: Chains of subgroups in symmetric groups. J. Algebra \textbf{127} (2), 340--352 (1989).

\bibitem{CRG16a} Castillo-Ramirez, A., Gadouleau, M.: Ranks of finite semigroups of one-dimensional cellular automata. Semigroup Forum \textbf{93}, no. 2, 347--362 (2016).

\bibitem{CRG16b} Castillo-Ramirez, A., Gadouleau, M.: On Finite Monoids of Cellular Automata. In: Cook, M., Neary, T. (eds.) Cellular Automata and Discrete Complex Systems. LNCS \textbf{9664}, 90--104, Springer International Publishing (2016).

\bibitem{CRG17} Castillo-Ramirez, A., Gadouleau, M.: Cellular automata and finite groups. Nat. Comput., First Online (2017).

\bibitem{CSC10} Ceccherini-Silberstein, T., Coornaert, M.: Cellular Automata and Groups. Springer Monographs in Mathematics, Springer-Verlag Berlin Heidelberg (2010).

\bibitem{Conrad} Conrad, K.: Dihedral Groups II. Retrieved from: http://www.math.uconn.edu/~kconrad/blurbs/grouptheory/dihedral2.pdf. 

\bibitem{GH87} Gomes, G.M.S., Howie, J.M.: On the ranks of certain finite semigroups of transformations. Math. Proc. Camb. Phil. Soc. \textbf{101}, 395--403 (1987).

\bibitem{G14} Gray, R.D.: The minimal number of generators of a finite semigroup. Semigroup Forum \textbf{89}, 135--154 (2014).

\bibitem{MN87} McIver, A., Neumann P.: Enumerating finite groups, Quart.J. Math. Oxford \textbf{38}, no. 4 (1987) 473--488.

\end{thebibliography}
\end{document}